\documentclass[11pt,twoside]{article}
\usepackage[utf8]{inputenc}
\usepackage[colorlinks=true]{hyperref}
\hypersetup{urlcolor=blue, citecolor=red, linkcolor=blue}

\usepackage{colortbl}

\usepackage{stmaryrd}

\usepackage{cancel}

\usepackage{graphicx}
\usepackage[active]{srcltx}

\usepackage{amsmath,amssymb}
\usepackage{bm}
\usepackage{enumerate}
\usepackage{amsthm}
\usepackage[all,2cell]{xy}
\usepackage{mathrsfs}
\usepackage{mathtools}

\usepackage{physics}

\makeatletter
\DeclareRobustCommand\widecheck[1]{{\mathpalette\@widecheck{#1}}}
\def\@widecheck#1#2{%
    \setbox\z@\hbox{\m@th$#1#2$}%
    \setbox\tw@\hbox{\m@th$#1%
       \widehat{%
          \vrule\@width\z@\@height\ht\z@
          \vrule\@height\z@\@width\wd\z@}$}%
    \dp\tw@-\ht\z@
    \@tempdima\ht\z@ \advance\@tempdima2\ht\tw@ \divide\@tempdima\thr@@
    \setbox\tw@\hbox{%
       \raise\@tempdima\hbox{\scalebox{1}[-1]{\lower\@tempdima\box
\tw@}}}%
    {\ooalign{\box\tw@ \cr \box\z@}}}
\makeatother

\usepackage{tikz-cd}
\usetikzlibrary{cd}

\usepackage{xcolor}

\hypersetup{urlcolor=blue, citecolor=red, linkcolor=blue}

\usepackage[a4paper, left=25mm, right=25mm, top=30mm, bottom=25mm]{geometry}

\usepackage{marginnote}

\usepackage{imakeidx}
\makeindex[intoc]

\usepackage[nottoc]{tocbibind}
\usepackage[colorinlistoftodos]{todonotes} 

\usepackage{import}

\newcommand{\dfn}[1]{\textbf{\textit{#1}}}

\theoremstyle{plain}
\newtheorem{theorem}{Theorem}[section]

\newtheorem{proposition}[theorem]{Proposition}

\newtheorem{lemma}[theorem]{Lemma}

\theoremstyle{definition}
\newtheorem{definition}[theorem]{Definition}

\newtheorem{example}[theorem]{Example}

\newcommand\restr[2]{{
  \left.\kern-\nulldelimiterspace 
  #1 
  \right|_{#2} 
}}

\newcommand{\R}{\mathbb{R}}
\renewcommand{\S}{\mathrm{S}}
\renewcommand{\d}{\mathrm{d}}

\newcommand{\Cinfty}{\mathscr{C}^\infty}
\newcommand{\T}{\mathrm{T}}

\newcommand{\Id}{\mathrm{Id}}

\newcommand{\Lie}{\mathscr{L}}
\newcommand{\vf}{\mathfrak{X}}

\newcommand{\totder}[2]{\frac{\mathrm{d} #1}{\mathrm{d} #2}}

\newcommand{\parder}[2]{\frac{\partial #1}{\partial #2}}

\DeclareMathAlphabet{\mathpzc}{OT1}{pzc}{m}{it}
\def\d{\mathrm{d}}

\DeclareMathOperator{\Sec}{Sec}

\def\tdLie{\mathbb{L}}

\def\tdBracket#1{\left\llbracket #1 \right\rrbracket}

\let\ds\displaystyle

\usepackage{fancyhdr}

\DeclareMathOperator{\longi}{Length}
\DeclareMathOperator{\ener}{Energy}

\DeclareMathSymbol{\shortmin}{\mathbin}{AMSa}{"39}

\newcommand{\dotnabla}{{%
    \mathchoice
    {
        \tikz[baseline=(nabla.base)]{%
            \node[inner sep=0pt, outer sep=0pt](nabla){$\nabla$};
            \node[inner sep=0pt, outer sep=0pt] at (nabla.center) [yshift=0.7pt,xshift=0.3pt] {\tiny$\circ$};
        }%
    }
    {
        \tikz[baseline=(nabla.base)]{%
            \node[inner sep=0pt, outer sep=0pt](nabla){$\nabla$};
            \node[inner sep=0pt, outer sep=0pt] at (nabla.center) [yshift=0.7pt] {\tiny$\circ$};
        }%
    }
    {
        \tikz[baseline=(nabla.base)]{%
            \node[inner sep=0pt, outer sep=0pt](nabla){\scriptsize$\nabla$};
            \node[inner sep=0pt, outer sep=0pt] at (nabla.center) [yshift=0.4pt] {\tiny$\circ$};
        }%
    }
    {
        \tikz[baseline=(nabla.base)]{%
            \node[inner sep=0pt, outer sep=0pt](nabla){\scriptsize$\nabla$};
            \node[inner sep=0pt, outer sep=0pt] at (nabla.center) [yshift=0.3pt] {\tiny$\circ$};
        }%
    }
}}

\usepackage{halloweenmath}
\renewcommand{\dotnabla}{{%
    \mathchoice
    {
        \tikz[baseline=(nabla.base)]{%
            \node[inner sep=0pt, outer sep=0pt](nabla){$\nabla$};
            \node[inner sep=0pt, outer sep=0pt] at (nabla.center) [yshift=0.7pt,xshift=0.3pt] {$\mathbat$};
        }%
    }
    {
        \tikz[baseline=(nabla.base)]{%
            \node[inner sep=0pt, outer sep=0pt](nabla){$\nabla$};
            \node[inner sep=0pt, outer sep=0pt] at (nabla.center) [yshift=0.7pt] {$\mathbat$};
        }%
    }
    {
        \tikz[baseline=(nabla.base)]{%
            \node[inner sep=0pt, outer sep=0pt](nabla){\scriptsize$\nabla$};
            \node[inner sep=0pt, outer sep=0pt] at (nabla.center) [yshift=0.4pt] {$\mathbat$};
        }%
    }
    {
        \tikz[baseline=(nabla.base)]{%
            \node[inner sep=0pt, outer sep=0pt](nabla){\scriptsize$\nabla$};
            \node[inner sep=0pt, outer sep=0pt] at (nabla.center) [yshift=0.3pt] {$\mathbat$};
        }%
    }
}}

\renewcommand{\dotnabla}{\mathcal{D}}

\DeclareMathOperator{\Flow}{F}

\begin{document}

\thispagestyle{empty}

{\LARGE\sffamily\raggedright
Time-dependent metrics and connections
}

\vspace{2em}

{\Large\raggedright\sffamily
    Xavier Gràcia
}\vspace{1mm}\newline
{\raggedright
    Department of Mathematics, Universitat Politècnica de Catalunya\\
    Barcelona, Catalonia, Spain\\
    e-mail: 
    \href{mailto:xavier.gracia@upc.edu}{xavier.gracia@upc.edu} --- orcid: 
    \href{https://orcid.org/0000-0003-1006-4086}{0000-0003-1006-4086}
}

\bigskip

{\Large\raggedright\sffamily
    Xavier Rivas
}\vspace{1mm}\newline
{\raggedright
    Department of Computer Engineering and Mathematics, Universitat Rovira i Virgili\\
    Avinguda Països Catalans 26, 43007 Tarragona, Spain\\
    e-mail: \href{mailto:xavier.rivas@urv.cat}{xavier.rivas@urv.cat} --- orcid: \href{https://orcid.org/0000-0002-4175-5157}{0000-0002-4175-5157}
}

\bigskip

{\Large\raggedright\sffamily
    Daniel Torres
}\vspace{1mm}\newline
{\raggedright
    Department of Mathematics, Universitat Politècnica de Catalunya\\
    Barcelona, Catalonia, Spain\\
    e-mail: \href{mailto:daniel.torres.moral@upc.edu}{daniel.torres.moral@upc.edu} --- orcid: \href{https://orcid.org/0000-0001-8355-3783}{0000-0001-8355-3783}
}

\vspace{2em}

{\large\raggedright\sffamily
20 January 2026
}

\vspace{2em}
{\large\bf\raggedright
Abstract
}\vspace{1mm}\newline
{\noindent
Time-dependent structures often appear in differential geometry, 
particularly in the study of non-autonomous differential equations on manifolds. 
One may study the geodesics associated with a time-dependent Riemannian metric
by extremizing the corresponding energy functional,
but also through the introduction of a more general concept of 
time-dependent covariant derivative operator.
This relies on the examination of connections on the product manifold $\mathbb{R} \times M$.
For these time-dependent covariant derivatives we explore the notions of
parallel transport, geodesics and torsion.
We also define the derivative of a one-parameter family of connections.
}
\bigskip

\medskip


\vspace{2em}
{\setcounter{tocdepth}{1}
\def\baselinestretch{1}
\small
\def\addvspace#1{\vskip 1pt}
\parskip 0pt plus 0.1mm
\tableofcontents
}

\pagestyle{fancy}

\fancyhead[C]{}
\fancyhead[RO]{Time-dependent metrics and connections}
\fancyhead[LO]{\thepage}
\fancyhead[RE]{\thepage}
\fancyhead[LE]{X. Gràcia, X. Rivas and D. Torres}

\fancyfoot[L]{}
\fancyfoot[C]{}
\fancyfoot[R]{}

\renewcommand{\headrulewidth}{0.1pt}
\renewcommand{\footrulewidth}{0pt}

\renewcommand{\headrule}{%
    \vspace{3pt}
    \hrule width\headwidth height 0.4pt
    \vspace{0pt}
}

\setlength{\headsep}{30pt}

\def\baselinestretch{1.1}
\normalsize
\section{Introduction}

Time-dependent structures often appear in differential geometry.
The simplest example of this is motivated by the study of non-autonomous differential equations on manifolds. 
A time-dependent vector field on a manifold~$M$ 
is a smooth one-parameter family $(X_t)$ of vector fields on it,
with $t \in \R$.
This family is more conveniently wrapped as a (smooth) map%
\footnote{We assume such a domain to simplify notation.
The domain of~$X$ could be more generally 
any non-empty open subset of $\R \times M$.}
$X \colon \R \times M \to \T M$
such that $X_t = X(t,\cdot)$ for each $t\in\R$.
In other words
$X(t,x) \in \T_xM$ for each $(t,x)$.
Thus, $X$ is a \emph{vector field along} 
the projection map 
$\rho_2 \colon \R \times M \to M$.
Such $X$ defines a non-autonomous differential equation on~$M$ given by
$\gamma'(t) = X(t,\gamma(t))$.
The study of this differential equation
can be performed by means of the  
\emph{suspension} 
\cite{MR_07}
(or autonomization)
of~$X$,
which is the vector field $\widetilde X$
defined on $\R \times M$ as 
$\widetilde X = \parder{}{t} + X$.
In this expression we identify, as usual,
the tangent space of a product
with the direct sum of the tangent spaces
---see Appendix~\ref{app:products} for more details and notation.
With the suspension at hand, 
it is easy to deduce properties of the time-dependent flow
of~$X$,
since it can be expressed in terms of the 
(time-independent) flow of the vector field $\widetilde X$.
Some details are given in Appendix~\ref{app:time-dep-flow}.

The main purpose of this paper is to study the notions of time-dependent metric and time-dependent connection on a manifold.
Riemannian metrics depending on a parameter have been widely used,
as for instance in the study of the Ricci flow
(see for instance 
\cite{Top_2006}
for a short overview)
that lead to the proof of the Poincar\'e conjecture.
Nevertheless, in this paper, we regard this parameter from a dynamical perspective,
by considering it to be the time parameter
of the differential equations associated with a metric or,
more generally, a connection.
From this perspective, 
time-dependent metrics can arise in at least two ways.
\begin{itemize}
\item 
There is a ``fixed'' space where effectively there is a kind of metric depending on time. 
\\
For instance, we can study
the optical path of a light ray in a medium with time-dependent refractive index;
the travel time of a car along a traffic network with variable traffic density;
the motion of a point particle with variable mass,
etc.
\item 
There is a ``fixed manifold'' along with a one-parameter family of embeddings into a Riemannian manifold;
the pullback of the Riemannian metric by these embeddings is a time-varying Riemannian metric on the manifold.
\\
For instance, this could be the case of a constraint manifold moving in an ambient space.
\end{itemize}

Both cases are well-known in analytical mechanics.
For instance,
in Newtonian mechanics one can deal with time-dependent forces,
and in Lagrangian mechanics one can deal with time-dependent Lagrangian functions.

Nevertheless there are some aspects that deserve a further study,
either from the mechanical viewpoint
(see for instance the discussion on virtual displacements in
\cite{RaSh_2006},
just to mention one example)
or from the geometrical viewpoint.
It is well known that a pseudo-Riemannian metric 
induces a canonical connection on a manifold,
the Levi-Civita connection,
which allows to define the notions of 
parallel transport and geodesic
(see for instance
\cite{Lee_18},
among the many textbooks on Riemannian geometry).
Trying to understand these constructions in the time-dependent setting will be one of our goals.
As we will see, 
a time-dependent metric $g$ on~$M$
does not define just some time-dependent Christoffel symbols,
but it also gives rise to a (time-dependent) endomorphism of the tangent bundle;
this endomorphism is directly related to the time-dependent tensor field~$\dot g$.

Thus, if we seek a notion of time-dependent covariant derivative,
we need to study this in a broader perspective.
So
we will study connections on the product manifold $\R \times M$,
and will try to extract what is the relevant information 
by comparing this with what one obtains in the Riemannian case.
So we will define a \emph{time-dependent covariant derivation operator}
$\dotnabla$,
which is not merely given by a one-parameter family of connections 
$^t\nabla$ on~$M$,
but needs some additional structures 
(two time-dependent endomorphisms $A$ and $B$ of the tangent bundle $\T M$,
and a time-dependent vector field $C$ on~$M$).
For this new notion 
we will study parallel transport and geodesics, and 
investigate what corresponds to the torsion.
In these calculations a kind of time-dependent Lie bracket will arise;
we will also take a look on this notion.
As a noteworthy fact, we will also show that the time derivative of a 
time-dependent connection is a (time-dependent) tensor field on the manifold.

The paper is organized as follows.
In Section~\ref{sec:time-dep-lie-bracket} we give the definitions for the time-dependent versions of Lie derivative of functions and vector fields.
In Section~\ref{sec:time-dep-metrics} we introduce the notion of
time-dependent (Riemannian) metric and the associated notion of length.
In Section~\ref{sec:variational} we perform a variational analysis of the length.
In Section~\ref{sec:Levi-Civita} we study the Levi-Civita connection of the suspension of a metric.
This leads us to Sections \ref{sec:connection-RxM}, \ref{sec:parallel-transport} and~\ref{sec:torsion},
where we study the more general case of 
a connection on the product $\R \times M$,
the notions of parallel transport and geodesics, 
and the torsion operator. In Section \ref{sec:example} we use the double pendulum with time-dependent masses to illustrate the theory.
Finally, there are a couple of appendices mainly to fix notations and 
recall some particular results about product manifolds and time-dependent flows.

Throughout the paper all manifolds are assumed to be smooth and paracompact,
and all maps are assumed to be smooth.
Einstein's summation convention over crossed repeated indices is understood.

\section{Time-dependent Lie bracket}
\label{sec:time-dep-lie-bracket}

In this section we study the notions of time-dependent Lie derivative and time-dependent Lie bracket.
This will not be used until Section~\ref{sec:torsion},
but we prefer to introduce it right now, since no metric or connection are needed.

Let $X$ be a time-dependent vector field on~$M$.
We denote by $s \mapsto \Flow_X(s,t,p)$
the integral curve of~$X$ with initial condition $(t,p)$;
$\Flow_X$ is the (time-dependent) flow of~$X$.
As it is explained in Appendix~\ref{app:time-dep-flow},
we can compute the variation of a function
$f\colon \R \times M \rightarrow \R$
along the flow of~$X$.
This leads us to the definition of the 
\dfn{time-dependent Lie derivative} of~$f$ with respect to~$X$ as
$$
\tdLie_Xf(t,p)  
= 
\lim_{\varepsilon\to 0}
\frac{ 
f\left(t{+}\varepsilon\,,\,\Flow_X^{t+\varepsilon,t}(p)\right) 
- f(t,p) }
{\varepsilon} \,,
$$
where we made use of the (local) diffeomorphisms 
$\Flow_X^{s,t} = \Flow_X(s,t,\cdot)$
defined by the flow.
It is easily checked that 
\begin{equation}
\tdLie_X f
= \parder{f}{t} + \Lie_X f
\,.
\end{equation}

In a similar way 
the tangent map 
$
\T \Flow_X^{t,s}
$
of the flow can be used to define the 
\dfn{time-dependent Lie derivative} of a time-dependent vector field $Y$ with respect to~$X$ as
$$
\tdLie_XY(t,p)
= 
\lim_{\varepsilon\to0} 
\frac{
\T_{\Flow_X(t+\varepsilon, t,p)}
\Flow_X^{t,t+\varepsilon}
\cdot 
Y(t+\varepsilon, \Flow_X(t{+}\varepsilon, t,p)) 
- 
Y(t,p)}
{\varepsilon}
\,.
$$
Another computation shows that
\begin{equation}
\tdLie_X Y = \dot Y + \Lie_X Y\,,
\end{equation}
where $\dot Y$ denotes the time derivative of the time-dependent vector field~$Y$ 
(see Appendix~\ref{app:products}).
A similar construction allows to differentiate any time-dependent tensor field on~$M$.

\medskip

The Lie bracket of vector fields can be defined in several ways:
algebraic (as a commutator of derivations), 
dynamical (as a Lie derivative), 
geometric (as a ``commutator'' of flows;
see 
\cite[Thm.\,3.16]{KMS_93}
or
\cite[Chap.\,5, Thm.\,16]{Spivak1},
for instance)
and even as a coordinate expression.
However, trying these procedures with time-dependent vector fields 
may lead to different or poor results.
We will choose the geometric definition.

\begin{proposition}\label{prop:timedep-Lie-bracket}
Let $X$, $Y$ be two time-dependent vector fields on~$M$,
and $(t,p) \in\R\times M$.
Consider
\vadjust{\kern -5mm}
\begin{align*}
c_1(\varepsilon) &= 
\Flow_X(t+\varepsilon, t, p)\,,
\\
c_2(\varepsilon) &= 
\Flow_Y(t+2\varepsilon, t+\varepsilon, c_1(\varepsilon))\,,
\\
c_3(\varepsilon) &= 
\Flow_X(t+\varepsilon, t+2\varepsilon, c_2(\varepsilon))\,,
\\
c_4(\varepsilon) &= \Flow_Y(t, t+\varepsilon, c_3(\epsilon))
\,.
\end{align*}
Then,
$c_4(0) = p$,  
$c_4'(0) = 0$, and 
\begin{equation}
\frac12 \, c_4''(0) = 
[X,Y]_{(t,p)} + \dot Y(t,p) - \dot X(t,p)
\,.
\end{equation}
\end{proposition}
\begin{proof}
It uses the properties of the time-dependent flow 
(Appendix~\ref{app:time-dep-flow}).
It is a delicate but not especially difficult calculation.
Alternatively one can compute the Taylor expansions in coordinates, up to order~2.
\end{proof}

A comment on this expression is in order. 
The tangent vector $c'(0)$ of a path in~$M$
is, of course, a vector tangent to~$M$, 
but this is not the case of $c''(0)$, which is tangent to $\T M$.
Nevertheless, when $c'(0) = 0$ then 
$c''(0)$ can be thought of as a tangent vector on~$M$.
This can be explained from the identification of a vector space
with its tangent space, 
from its interpretation as a point derivation on functions,
or more simply from its transformation under coordinate changes.
See also the references mentioned above.

\medskip
In view of the preceding result we propose the following geometric definition:

\begin{definition}
The 
\dfn{time-dependent Lie bracket} of two time-dependent vector fields $X,Y$ is
the time-dependent vector field
\begin{equation}
\label{eq:tdbracket}
\tdBracket{X,Y} = 
[X,Y] + \dot Y - \dot X
\,.
\end{equation}
\end{definition}

This definition is related to 
the suspensions $\widetilde{X}$ and $\widetilde{Y}$ 
and to the usual Lie bracket in $\R\times M$ in a natural way.
A trivial computation shows:

\begin{proposition}
Let $X, Y$ be two time-dependent vector fields on~$M$,
and $\widetilde X, \widetilde Y$ their suspensions to $\R \times M$.
Then
\begin{equation}
\tdBracket{X, Y} = [\widetilde X, \widetilde Y] 
\,.
\end{equation}
\qed
\end{proposition}

Another technical comment is in order.
As it is explained in Appendix~\ref{app:products},
we have a decomposition
$\vf(\R \times M) = \vf(\rho_1) \oplus \vf(\rho_2)$,
that is,
a vector field $\widehat X$ on the product is the sum of 
a vector field along 
$\rho_1 \colon \R \times M \to \R$
(which we could call its ``horizontal part''
with respect to the trivial fibration~$\rho_1$),
and 
a vector field along 
$\rho_2 \colon \R \times M \to M$ 
(a time-dependent vector field, the ``vertical part'').
The formula in the proposition says that the Lie bracket of two suspensions
is purely ``vertical''.

Finally, notice that
\begin{equation}
\tdBracket{X, Y} =  \tdLie_{X}Y - \dot X
,
\end{equation}
therefore the time-dependent Lie bracket and Lie derivative are different.
As well as the ordinary Lie bracket, the anticommutative property is still conserved,
in contrast with the time-dependent Lie derivative.

\section{Time-dependent Riemannian metrics}
\label{sec:time-dep-metrics}

\begin{definition}
A \dfn{time-dependent Riemannian metric} 
on a manifold~$M$
is a smooth family of Riemannian metrics $(g_t)$ on~$M$,
namely a smooth map 
$g \colon \R \times M \to \T^* M\otimes\T^* M$ 
such that 
$g_t = g(t,\cdot)$ is a Riemannian metric for every $t\in \R$. 
In other words, 
the map $g$ is a positive-definite symmetric 
2-covariant tensor field 
along the projection 
$\rho_2 \colon \R \times M \to M$.
\end{definition}

Associated with the metrics $g_t$
there are a number of canonical constructions,
as for instance the curvature tensors $R_t$;
these constitute the time-dependent curvature tensor,
which we can understand as a tensor field along the projection~$\rho_2$. 

We restrict ourselves to the Riemannian case,
but of course many of the following considerations 
apply to the pseudo-Riemannian case,
sometimes with certain restrictions.

\smallskip 

It seems natural to define the 
\dfn{length}
of a smooth path 
$\gamma \colon I \to M$
as the integral 
\begin{equation}
\longi[\gamma] 
=
\int_I \sqrt{g_t(\gamma'(t),\gamma'(t))} \,\d t
\,.
\end{equation}

\subsection*{Metric induced by a one-parameter family of embeddings}

A special case of interest of time-dependent metrics 
is provided by the following setting.
Consider a Riemannian manifold 
$(\widetilde M,\widetilde g)$
and a smooth one-parameter family of embeddings 
$j_t = j(t,\cdot) \colon M \to \widetilde M$.
Then we obtain a time-dependent metric on~$M$ 
by defining 
\begin{equation*}
g_t = j_t^*(\widetilde g)\,.
\end{equation*}
Given a path $\gamma \colon I \to M$,
composing it with these embeddings we obtain a path 
$\widetilde\gamma \colon I \to \widetilde M$
given by
$\widetilde\gamma(t) = j_t(\gamma(t))$.
So $\gamma$ has an 
\dfn{embedded length} 
\begin{equation*}
\longi_\mathrm{emb}[\gamma] 
=
\longi_{\widetilde g}[\widetilde\gamma] 
=
\int_I \sqrt{\widetilde g(\widetilde\gamma'(t),\widetilde\gamma'(t))} \,\d t
\,.
\end{equation*}
This ``length'' is quite different from the previously length,
since for instance a constant path $\gamma$ in~$M$ 
may not be constant when embedded in~$\widetilde M$,
and conversely.
Let us show all of this in a simple example.

\begin{example}
Consider the unit circle $M = \S^1$ endowed with the usual angular coordinate~$\theta$, and $\widetilde M = \R^2$ with its usual Riemannian metric which reads, in Cartesian coordinates,
$\widetilde g = \d x\otimes\d x + \d y\otimes\d y$.

\begin{itemize}
\item 
Take the family of embeddings $(j_t)$ given by
$j_t \colon \S^1 \to \R^2$ 
given by 
$j_t(p) = tp$.
In the aforementioned coordinates they read
$\theta \mapsto (t \cos\theta, t\sin\theta)$.
Then $g_t = t^2 \d\theta\otimes\d\theta$.

Consider the path 
$\gamma(t) = (\cos 2\pi t,\sin 2\pi t)$,
for $t \in I = [0,1]$.
One easily computes 
$\norm{\gamma'(t)}_{g_t} = 2\pi t$
and therefore
$
\ds 
\longi[\gamma] =
\int_0^1 \!\!\! 2\pi t \,\d t = \pi
$,
whereas
$\widetilde\gamma(t) = (t \cos 2\pi t, t\sin 2\pi t)$
yields
$
\ds
\longi_\mathrm{emb}[\gamma] =
\int_0^1 \!\!(1+4\pi^2t^2)^{1/2} \,\d t
\approx
3.383044
$,
slightly bigger. 

Instead, if we take the constant path 
$\delta(t)=(1,0)$
for $t \in [0,1]$,
obviously 
$\longi[\delta] = 0$,
whereas 
$\longi_\mathrm{emb}[\delta] = 1$.

\item
Now consider the embeddings 
$i_t \colon \S^1 \to \R^2$ 
given by the clockwise rotation of angle~$t$.
The induced metrics are $g_t = \d\theta\otimes\d\theta$.
If we take the same path $\gamma$ as before,
$\longi[\gamma] = 2\pi$,
whereas 
$\longi_\mathrm{e}[\gamma] = 0$,
since the embedded path is now constant.
\end{itemize}
\end{example}

\subsection*{The suspension of a metric}

Let $(g_t)$ be a time-dependent metric on~$M$.
On the product manifold
$\R \times M$,
we consider the \dfn{suspension}
$$
\widehat g|_{(t,x)} = 
\d t \otimes \d t |_{(t,x)} + \rho_2^*(g_t)|_{(t,x)}
\,,
$$
where $\d t \otimes \d t$ 
is the pullback of the canonical Riemannian metric of~$\R$ 
with respect to the projection
$\rho_1 \colon \R \times M \to \R$.

Notice that $g_t = j_t^*(\widehat g)$, where 
$j_t(x) = (t,x)$.
Thus, any time-dependent metric 
can be obtained as the pullback of a metric $\widehat g$
through a one-parameter family of embeddings,
as in the preceding examples.
However, not always this $\widehat g$ will have geometric interest.

\section{Variational analysis}
\label{sec:variational}

In Riemannian geometry, geodesics arise as critical paths of the
energy functional,
and their reparametrizations
are the critical paths of the length functional.
So, let us study these same functionals for the time-dependent case:
\begin{align}
    \ener[\gamma]
    & = \int_a^b T(t,\gamma'(t)) \, \d t
    = \int_a^b \frac12 g_t(\gamma'(t), \gamma'(t)) \, \d t\,,
    \\
    \longi[\gamma]
    & = \int_a^b L(t,\gamma'(t)) \, \d t
    = \int_a^b \sqrt{2 T\,} \, \d t
    = \int_a^b \sqrt{g_t(\gamma'(t), \gamma'(t))} \, \d t
    \,.
\end{align}

The critical paths for these functionals 
are the solutions of their corresponding Euler--Lagrange equations. 
Let us compute them first for the energy functional.
We consider coordinates $(x^i)$ on~$M$,
we denote by $(x^i, u^i)$ 
the corresponding natural coordinates of the tangent bundle,
and write
$g = g_{ij}(t, x)\,\d{x^i}\otimes\d{x^j}$.
Thus, the Euler--Lagrange equations in coordinates read
$$
\frac\d{\d t}\left(\parder{T}{u^\ell}\right) - \parder{T}{x^\ell}
= 
\parder{g_{\ell i}}{t} \dot\gamma^i +
g_{\ell i} \ddot\gamma^i + 
\frac12 
\left(
  \parder{g_{i\ell}}{x^j}
  + \parder{g_{j\ell}}{x^i}
  - \parder{g_{ij}}{x^\ell}
\right) \dot\gamma^i \dot\gamma^j  
= 
0
\,.
$$
Since the metric is non-degenerate
we can multiply this expression by the 
inverse matrix $g^{k\ell}$ of the metric,
thus obtaining
\begin{equation}
\label{eq:geodesic-coordinates}
\ddot\gamma^k + 
\Gamma_{ij}^k \dot\gamma^i \dot\gamma^j + 
g^{k\ell} \parder{g_{\ell i}}{t} \dot\gamma^i
= 
0\,,
\end{equation}
where $\Gamma_{ij}^k$ are the (time-dependent) Christoffel symbols 
of the Levi-Civita connection $^t\nabla$ 
of each metric $g_t$.

The first two terms seem to correspond to the equation of geodesics,
that is,
the covariant derivative 
$\nabla_{\!t} \gamma'(t)$
of the vector field $\gamma'$ along~$\gamma$,
but where the coefficients are now time-dependent since they are
the Christoffel symbols of 
a one-parameter family of connections~$\nabla$.

There is an additional time-dependent term.
To identify it,
we make use of the musical isomorphisms 
$$ 
G_t \colon \T M \to \T^*M
\,,\quad 
v \mapsto g_t(v,\cdot)
$$
associated with~$g_t$,
and their inverses 
$G_t^{-1} \colon \T^*M \to \T M$.
Also we will use the dot notation for the time-derivative,
and will mostly omit the $t$ parameter in the equations.
Notice, however, that 
$\dot G_t \colon \T M \to \T^*M$ 
may no longer be an isomorphism.
With these conventions, we have:

\begin{proposition}
Let $\gamma\colon I \to M$ be a path.
Then $\gamma$ is a critical path of the energy functional if and only if
\begin{equation}
\label{eqn:geodesic_of_g}
\nabla_t \gamma' + 
G^{-1} \cdot \dot G \cdot \gamma'
= 
0\,.
\end{equation}
\end{proposition}

\begin{definition}
A \dfn{geodesic of the time-dependent metric}~$g$  
is a path $\gamma$ satisfying 
eq.~\eqref{eqn:geodesic_of_g}.
\end{definition}

Analogously, we have the following result for the length functional:
\begin{proposition}
Let $\gamma\colon I \to M$ be a path.
Then $\gamma$ is a critical path of the length functional if and only if
\begin{equation}
\nabla_t \gamma' + 
G^{-1} \cdot \dot G \cdot \gamma'
- \frac1{2 T} \frac{\d T}{\d t} \, \gamma' 
= 0
\,,
\end{equation}
where the kinetic energy $T$ is evaluated along~$\gamma$.   
\end{proposition}

We omit the proofs of both formulas,
which are similar to those of the time-independent case.

Another point to consider is the conservation of the kinetic energy.

\begin{proposition}
Let $\gamma\colon I\to M$ be a path. Then
$$
\frac\d{\d t}T 
= \frac12 \dot g(\gamma',\gamma') + g(\nabla_t \gamma',\gamma')
= 
g(\nabla_t \gamma' + G^{-1} \cdot \dot G \cdot \gamma', \gamma')
-
\frac12 \dot g(\gamma',\gamma') 
,
$$
where the kinetic energy is evaluated along~$\gamma$.
\end{proposition}

As a consequence, 
for a critical path of the energy functional, 
the kinetic energy is conserved if and only if
$\gamma'$ is an isotropic vector of the 
(possibly degenerate) symmetric 2-covariant tensor field~$\dot g$.

In the time-independent case,
the kinetic energy along an arbitrary path is conserved if and only if 
$\gamma'$ is orthogonal to 
$\nabla_t \gamma'$.

\section{The Levi-Civita connection associated with the suspension of a metric}
\label{sec:Levi-Civita}

We consider again the suspension
$
\widehat g = 
\d t \otimes \d t + \rho_2^*(g_t)
$
of a time-dependent metric.
We will also denote by $x^0=t$ 
the natural coordinate of~$\R$,
and represent by $(x^\mu) = (x^0;x^i)$
coordinates in the product manifold 
$\widehat M = \R \times M$.
If $g_t$ has matrix representation
$\mathbf{G}_t(x) = (g_{ij}(t,x))$ at a point~$x$,
the matrix representation of $\widehat g$ is 
$
\left(\begin{smallmatrix}
  1 & 0 \\
  0 & \mathbf{G}_t(x)
\end{smallmatrix}\right)
$.

The $g_t$, as well as the suspension~$\widehat g$,
have associated Levi-Civita connections,
with Christoffel's symbols $\Gamma_{ij}^{k}$ 
(we omit the~$t$ for the sake of clarity)
and
$\widehat \Gamma_{\mu\nu}^{\rho}$, respectively.
An easy computation yields 
$$
\widehat \Gamma_{ij}^{k} =
\Gamma_{ij}^{k}\,,
\qquad 
\widehat \Gamma_{ij}^{0} = 
-\frac12 \parder{g_{ij}}{t}\,,
\qquad 
\widehat \Gamma_{i0}^{k} = 
\widehat \Gamma_{0i}^{k} = 
\frac12 g^{k\ell} \parder{g_{i\ell}}{t}\,,
$$ 
while the others are zero:
$
\widehat \Gamma_{i0}^{0} = 
\widehat \Gamma_{0i}^{0} =
\widehat \Gamma_{00}^{\mu} =
0
$.

We can use these Christoffel's symbols to compute 
covariant derivatives,
as for instance 
$$
\widehat\nabla_{\parder{}{x^i}} \parder{}{x^j} =
\Gamma_{ij}^k \parder{}{x^k} - \frac12 \parder{g_{ij}}{t} \parder{}{t}
\,,
\qquad
\widehat\nabla_{\parder{}{t}} \parder{}{x^j} =
\frac12 g^{k\ell} \parder{g_{\ell j}}{t} \parder{}{x^k}
\,,
\qquad
\widehat\nabla_{\parder{}{t}} \parder{}{t} = 0
\,.
$$
Also we can compute the covariant derivative 
$\widehat\nabla_{\!s} \widehat w$
of a vector field $\widehat w$
along a path $\widehat\gamma$ in~$\widehat M$
(here we use $s$ for the parameter, 
since for the time being we don't suppose it to coincide with the $t$ coordinate).
In particular, one can compute
the ``acceleration''
$\widehat\nabla_{\!s} \widehat\gamma'$
of a path
$\widehat\gamma(s) = (\gamma^0(s);\gamma^i(s))$,
whose components are
$$
\ddot\gamma^0 - 
\frac12 \parder{g_{ij}}{t} \dot\gamma^i \dot\gamma^j
\qquad
\text{and}
\qquad 
\ddot\gamma^k +
\Gamma_{ij}^k \dot\gamma^i \dot\gamma^j +
g^{k\ell} \parder{g_{i\ell}}{t} \dot\gamma^0 \dot\gamma^i
\,.
$$

We are especially interested in the case where
$\widehat\gamma$ is indeed a section of 
$\R \times M \to \R$,
namely  
$\widehat\gamma(t) = (t;\gamma^i(t))$.
Then the preceding components become
$$
- \frac12 \parder{g_{ij}}{t} \dot\gamma^i \dot\gamma^j
\qquad
\text{and}
\qquad 
\ddot\gamma^k +
\Gamma_{ij}^k \dot\gamma^i \dot\gamma^j +
g^{k\ell} \parder{g_{i\ell}}{t} \dot\gamma^i
\,,
$$
so we can write
\begin{equation}
\widehat\nabla_{\!t} \widehat\gamma' =
-\frac12 \dot g(\gamma',\gamma') \,\parder{}{t} 
+
\left(
\nabla_t \gamma' + 
G^{-1} \cdot \dot G \cdot \gamma'
\right)
.
\end{equation}
In the last equation we recognize the equation of the geodesics of a time-dependent metric,
eq.~\eqref{eqn:geodesic_of_g}.
So we arrive to the following geometric interpretation:
\begin{lemma}
$\widehat\gamma$ is a geodesic of~$\widehat\nabla$ 
if and only if
$\gamma'$ is isotropic with respect to~$\dot g$
and 
$\gamma$ is a geodesic of the time-dependent metric~$g$.
\end{lemma}

\section{Connections on \texorpdfstring{$\R \times M$}{}. Time-dependent connections}
\label{sec:connection-RxM}

Inspired by the results of the preceding section,
let us consider a connection 
$\widehat\nabla$ on the manifold 
$\widehat M = \R \times M$,
with Christoffel's symbols 
$\widehat\Gamma_{\mu\nu}^{\rho}$,
with which we can compute 
covariant derivatives 
$
\widehat\nabla_{\widehat X} {\widehat Y}
$
of vector fields on~$\widehat M$.
Here we use notations similar as those of the preceding section,
e.g.\ $x^0=t$, $(x^\mu) = (x^0;x^i)$.

Consider two vector fields on $\R \times M$
$$
\widehat X = 
f^0 \parder{}{t} + f^i \parder{}{x^i} =
f^0 \parder{}{t} + X
\,,\qquad 
\widehat Y = 
g^0 \parder{}{t} + g^i \parder{}{x^i} =
g^0 \parder{}{t} + Y
$$ 
and let us also separate the time and space indices in 
the covariant derivative:
\begin{multline*}
\widehat\nabla_{\widehat X} {\widehat Y}
=
\left( 
f^0 \dot g^0 + 
\Lie_X g^0 +
\widehat\Gamma_{00}^0 f^0 g^0+ 
\widehat\Gamma_{i0}^0 f^i g^0+ 
\widehat\Gamma_{0j}^0 f^0 g^j +
\widehat\Gamma_{ij}^0 f^i g^j 
\right)
\parder{}{t}
\\
+
\left( 
f^0 \dot g^k +
\Lie_X g^k + 
\widehat\Gamma_{ij}^k f^i g^j
\right) 
\parder{}{x^k}
+
\left( 
\widehat\Gamma_{00}^k f^0 g^0 + 
\widehat\Gamma_{i0}^k f^i g^0 + 
\widehat\Gamma_{0j}^k f^0 g^j 
\right) 
\parder{}{x^k} \,.
\end{multline*}

In this expression there is a number of terms that seemingly have  geometric interpretations.
One may notice,
by means of a change of coordinates in~$M$,
that they actually are geometric objects along the projection to~$M$,
but indeed a geometric proof can be provided.
This is the contents of the next theorem.

\begin{theorem}
\label{th:nabladotCharacterization}
For two vector fields
$\widehat X = f^0 \parder{}{t} + X$
and 
$\widehat Y = g^0 \parder{}{t} + Y$
on $\R \times M$,
the covariant derivative can be expressed as
\begin{multline}
\widehat\nabla_{\widehat X} {\widehat Y} 
=
\left(\strut 
f^0 \dot g^0 +
\Lie_X g^0 +
\lambda f^0 g^0 + 
\alpha (X) g^0 + 
\beta(Y) f^0 + \varepsilon(X,Y) 
\right) \parder{}{t}
\:+
\\
+
\left(
f^0 \dot Y + \nabla_X Y + 
C f^0 g^0 + 
A (X) g^0 + 
B (Y) f^0 
\right)
,
\end{multline}
where 
\begin{itemize}
\itemsep 0pt
\item 
$\lambda$ is a time-dependent function,
\item
$\alpha$ and~$\beta$ are time-dependent 1-forms,
\item 
$\varepsilon$ is a time-dependent 2-covariant tensor field,
\item 
$C$ is a time-dependent vector field,
\item 
$A,B$ are time-dependent endomorphisms of $\T M$,
and 
\item 
$\nabla \equiv {}^t\nabla$ is a time-dependent connection on~$M$.
\end{itemize}
\end{theorem}

\begin{proof}
The proof follows from writing the covariant derivative for four particular values of the pair
$(\widehat X,\widehat Y)$,
given by either by $\parder{}{t}$, 
or ``vertical'' vector fields $X,Y$.
Then one sees that these objects are defined geometrically as follows:
\begin{align*}
& \lambda =
\left\langle  \d t,
\T(\rho_1) \circ \widehat\nabla_{\parder{}{t}} {\parder{}{t}}
\right\rangle
,
\\
& \alpha(X) = 
\left\langle  \d t,
\T(\rho_1) \circ \widehat\nabla_{X} {\parder{}{t}}
\right\rangle
\,,
\\
& \beta(Y) = 
\left\langle  \d t,
\T(\rho_1) \circ \widehat\nabla_{\parder{}{t}} Y 
\right\rangle
\,,
\\
& \varepsilon(X,Y) = 
\left\langle  \d t,
\T(\rho_1) \circ \widehat\nabla_{X} Y 
\right\rangle
\,,
\\
& C = 
\T(\rho_2) \circ \widehat\nabla_{\parder{}{t}} {\parder{}{t}}
\,,
\\
& A (X) = 
\T(\rho_2) \circ \widehat\nabla_{X} \parder{}{t}
\,,
\\
& B (Y) = 
\T(\rho_2) \circ \widehat\nabla_{\parder{}{t}} Y - \dot Y
\,,
\\
& \nabla_X Y =
\T(\rho_2) \circ \widehat\nabla_X Y 
\,.
\end{align*}
From this it is clear that 
$\lambda$ is a function and $C$ is a vector field,
and that 
$\alpha$ is a differential 1-form and 
$A$ is an endomorphism
(all of them time-dependent).
An easy calculation shows that 
$\beta$, $B$ and $\epsilon$ are also well defined,
and finally that 
$\nabla_X Y$ is a covariant derivative on~$M$
depending upon the $t$ parameter.
\end{proof}

So, a connection $\widehat\nabla$
is equivalently described by the data 
$(\lambda,\alpha,\beta,\varepsilon,C,A,B,\nabla)$.
Then one can characterize properties of~$\widehat\nabla$
in terms of these elements.
For instance, 
let us compute the covariant derivative with respect to the \emph{suspensions} $\widetilde X,\widetilde Y$
of time-dependent vector fields $X,Y$:
$$
\widehat\nabla_{\widetilde X} {\widetilde Y} 
=
\left(\strut 
\lambda + 
\alpha (X) + 
\beta(Y) + \varepsilon(X,Y) 
\right) \parder{}{t}
+
\left(
\dot Y + \nabla_X Y + 
C + 
A (X) + 
B (Y)
\right)
,
$$
From this it is clear that 
the covariant derivation of suspension vector fields
is another suspension iff 
$\lambda=1$ and $\alpha$, $\beta$ and $\varepsilon$ are zero.

Nevertheless, if we start from 
time-dependent vector fields, 
it seems natural to project this covariant derivative to $\T M$
to extract its ``vertical'' part:
$$
\T(\rho_2) \circ \widehat\nabla_{\widetilde X} {\widetilde Y}
=
\dot Y + \nabla_X Y + C + A(X) + B(Y)
\,,
$$
which is another time-dependent vector field on~$M$.
This motivates the following definition:

\begin{definition}
A \dfn{time-dependent covariant derivation operator} 
$\dotnabla$ on~$M$
is a map 
that sends two time-dependent vector fields $X,Y$ to another one,
$(X,Y) \mapsto \dotnabla_X Y$,
of the form 
\begin{equation}
\label{eq:tdcovariant}
\dotnabla_X Y =
\dot Y + \nabla_X Y + C + A(X) + B(Y)\,,
\end{equation}
where $\nabla$ is a time-dependent connection on~$M$,
$C$ is a time-dependent vector field,
and $A$ and~$B$ are time-dependent endomorphisms of $\T M$.
Sometimes we will write 
$\dotnabla = (\nabla,C,A,B)$.
\end{definition}

However, this map lacks the usual properties of a covariant derivation of the time-independent case.
For instance,
let us define the time-dependent covariant derivative of a time-dependent function
$f \colon \R \times M \to \R$ as 
$
\dotnabla_X f = 
\parder{f}{t} + \Lie_X f
$.
Then
$$
\dotnabla_X (f Y) = 
f \dotnabla_X Y + (\dotnabla_X f) Y + (1 {-} f) (C + A(X))
\,,
$$
from which we see that
the operators $\dotnabla_X$ satisfy the Leibniz rule on products $fY$
if, and only if, 
$C = 0$ and $A = 0$.

Similarly,
one sees that
$\dotnabla_X Y$ is not $\Cinfty(\R {\times} M)$-linear in~$X$.
Indeed, $\dotnabla_0$ is never the zero operator.
This seemingly strange behaviour was not unexpected:
while the time-dependent vector fields form a vector space,
their suspensions are rather an affine space.

\subsection*{The time-derivative of a time-dependent connection}

It is well known that the Christoffel symbols $\Gamma_{ij}^k$ 
of a connection $\nabla$ are \emph{not} the components of a tensor field.
This can be seen in several ways, for instance by looking at the 
change of coordinates, or by realizing that connections on a manifold form an infinite-dimensional affine space, but not a vector space.

What is remarkable, though, is that, 
if instead of a connection $\nabla$ 
indeed we have a 
one-parameter family of connections, say $^t\nabla$,
then its ``time-derivative'' \emph{is} a tensor field.
\begin{proposition}
Let $^t\nabla$, $t \in \R$, be a one-parameter family of connections on a manifold~$M$.
Then the equation 
\begin{equation}
\dot\Gamma (X,Y) =
\lim_{\varepsilon \to 0} 
\frac{ ^{t+\varepsilon}\nabla_X Y - {}^t\nabla_X Y}{\varepsilon}
\,,
\end{equation}
where $X,Y$ are arbitrary vector fields on~$M$,
defines a time-dependent $(2,1)$-tensor field on~$M$.
\end{proposition}
\begin{proof}
The proof is simply a computation that shows that the components of $\dot\Gamma$
are indeed the time-derivatives 
$\dot\Gamma_{ij}^k$
of the Christoffel symbols.
In other words,
$$
\dot\Gamma = 
\dot\Gamma_{ij}^k \: \d x^i \otimes \d x^j \otimes \parder{}{x^k}
\,.
$$
Notice that the tensor character of $\dot\Gamma$
is ensured by the fact that the numerator of the fraction is the difference of two connections,
which is a tensor field.
\end{proof}
\begin{definition}\label{def:dotGamma}
We call the time-dependent tensor field $\dot\Gamma$ the 
\dfn{time-derivative} of the time-dependent connection. 
\end{definition}

\section{Parallel transport and geodesics}
\label{sec:parallel-transport}

As in the preceding section,
we start with the time-extended manifold $\widehat M = \R \times M$
provided with a connection $\widehat\nabla$,
and shortly after we will use a time-dependent covariant derivation operator~$\dotnabla$.

\subsection*{Covariant derivative along a path}

If we have a path 
$\widehat\gamma = (\gamma^0,\gamma) \colon I \to \widehat M$,
the pullback of $\widehat \nabla$ by~$\widehat\gamma$
yields a connection on the pullback vector bundle 
$\widehat\gamma^*(\T \widehat M)$,
whose sections are the vector fields along~$\widehat\gamma$.
With the pullback connection one can compute the covariant derivative 
$\widehat\nabla_s \widehat w$
of
$\widehat w \in \vf(\widehat\gamma)$ 
with respect to the canonical vector field
$\totder{}{s}$
of~$I \subset \R$.
Then, using the notations of the preceding section,
and expressing 
$\widehat w$ 
as 
$\widehat w(s) = 
w^0(s) \left. \totder{}{t} \right|_{\gamma^0(s)}
+
w(s)
\in 
\T_{\gamma^0(s)} \R \oplus \T_{\gamma(s)}M
$,
we can organize the result as 
\begin{multline}
\qquad
\widehat\nabla_{\!s} \,\widehat w 
=
\left(\strut
\dot w^0
+ \lambda \dot\gamma^0 w^0
+ \alpha(\gamma') w^0
+ \beta(w) \dot\gamma^0 
+ \varepsilon(\gamma', w)
\right) \parder{}{t} \circ \widehat\gamma
\\ 
+
\nabla_{\!s} \,w + 
\dot\gamma^0 w^0 \,C +
w^0 A(\gamma') +
\dot\gamma^0 B(w) 
\,.
\qquad
\end{multline}
In this expression 
$\nabla_{\!s} w$
is the covariant derivation of $w \in \vf(\gamma)$
with respect to the time-parametrized family of connections $\nabla$,
which means that, in coordinates, 
$$
\nabla_{\!s} w (s) 
=
\left(
\dot w^k(s) + 
\Gamma_{ij}^k(s,\gamma(s)) \,\dot\gamma^i(s) \,w^j(s) 
\right)
\left.
\parder{}{x^k} 
\right|_{\gamma(s)}
.
$$

Nevertheless, for our purposes,
the path should be a section of the projection 
$\R \times M \to \R$,
that is,
$\gamma^0(s)=s$.
In equivalent terms,
we consider a path $\gamma \colon I \to M$,
and then we extend it to
$\widehat\gamma(t) = (t,\gamma(t))$.

In the same way, 
we consider $w\in\vf(\gamma)$, a vector field along~$\gamma$,
and its associated suspension
$\widetilde w(t) = (1,w(t))$,
which is a vector field along~$\widetilde \gamma$.
Then the preceding covariant derivative becomes 
$$
\widehat\nabla_{\!t} \,\widehat w 
=
\left(\strut
\lambda  
+ \alpha(\gamma') 
+ \beta(w) 
+ \varepsilon(\gamma', w)
\right) \parder{}{t} \circ \widehat\gamma
+
\left(\strut
\nabla_{\!t} \,w + 
C +
A(\gamma') +
B(w) 
\right)
\,.
$$
As in the preceding section we retain its vertical part,
and use it as a definition:
\begin{definition}
Given a time-dependent covariant derivative operator 
$\dotnabla = (\nabla,C,A,B)$ on~$M$,
the   
\dfn{time-dependent covariant derivative of a vector field along a path}
is the operator $\dotnabla_{t}$ defined as
\begin{equation}
\dotnabla_{t} w 
= 
\nabla_{\!t} w + 
C +
A(\gamma') +
B(w)\,,
\end{equation}
for any $w \in \vf(\gamma)$.
\end{definition}

\subsection*{Parallel transport}

For the usual time-independent connections,
a vector field
$\widehat w \in \vf(\widehat\gamma)$ 
is said to be parallel (along $\widehat\gamma$)
when
$
\widehat\nabla_{\!t} \widehat w 
= 0
$.
This is a time-dependent linear homogeneous differential equation for
$(w^0,w)$
and has a unique solution for a given initial condition
$(w^0(t_\circ),w(t_\circ))$.

Following our preceding discussions,
we will extend this notion to the time-dependent setting in a similar way:
\begin{definition}
A vector field $w$ along a path~$\gamma$ is 
\dfn{parallel} 
with respect to a time-dependent covariant derivative operator when
$
\dotnabla_{t} w = 0
$.
\end{definition}
This means that $w$ satisfies the non-homogeneous linear differential equation 
\begin{equation}
\nabla_{\!t} w + C + A(\gamma') + B(w) = 0
\,.
\end{equation}
In coordinates this reads
\begin{equation*}
\dot w^k + 
\Gamma_{ij}^k \dot\gamma^i w^j +
C^k + A_{\:i}^k \dot\gamma^i + B_{\:i}^k w^i
=
0
\,.
\end{equation*}

\subsection*{Geodesics}

\begin{definition}
A path $\gamma$ is a \dfn{geodesic} of a time-dependent covariant derivation operator if 
$
\dotnabla_{t} \gamma' = 0
$.
\end{definition}
That is,
\begin{equation}
\nabla_{\!t} \gamma' + 
C +
(A + B) (\gamma')
= 
0
\,.
\end{equation}
In coordinates this reads
$$
\ddot \gamma^k + 
\Gamma_{ij}^k \dot\gamma^i \dot\gamma^j +
C^k + (A_{\:i}^k + B_{\:i}^k) \dot\gamma^i
=
0
\,.
$$
(Remember always that the $\Gamma$, $A$, $B$, $C$ are evaluated along $(t,\gamma(t))$.)

\medskip
In the metric case,
that is when $\widehat\nabla$ is the Levi-Civita connection of the suspension $\widehat g$,
we have $C=0$ and $A = B$, 
with 
$A = \frac12 \: G^{-1} \cdot \dot G$.
In this case we recover
eq.~\eqref{eqn:geodesic_of_g}.

\section{The torsion operator}
\label{sec:torsion}

The torsion tensor of a connection on a manifold
is defined as
$T^\nabla(X,Y) = \nabla_XY - \nabla_YX -[X,Y]$.
When providing a time-dependent version for the torsion,
we cannot expect it to preserve some properties like, for instance,
to be actually a tensor field.
So, in the same way we did for the Lie bracket in Section~\ref{sec:time-dep-lie-bracket},
in order to preserve the original concept of the torsion,
we will rely on its geometric interpretation.

Such interpretation is as follows:
given a point $p_0$ and two tangent vectors $v_0,w_0\in \T_{p_0} M$,
we perform the following four-step procedure:
\begin{itemize}
\item 
First, we parallel transport $v_0$ and $w_0$ along the geodesic passing by $p_0$ with direction $v_0$ during $\varepsilon$ time units.
Let us call the resulting point and tangent vectors $p_1$ and $v_1, w_1\in \T_{p_1}M$.
\item 
Second, we parallel transport these vectors along the geodesic defined by $p_1$ and $w_1$ during $\varepsilon$ time units.
\item 
Then, the same for the geodesic defined by $p_2$ and $v_2$ during $-\varepsilon$ time units.
\item 
Finally, for the geodesic defined by $p_3$ and $w_3$ for $-\varepsilon$ time units.
\end{itemize}
The resulting point, $c(\varepsilon) = p_4$, as a function of~$\varepsilon$,
describes a smooth path.
Its tangent vector at $\varepsilon=0$ is zero,
therefore half of its second derivative is identified with a tangent vector.
This vector corresponds (up to a sign) to the torsion tensor of the connection applied to the vectors $v,w$.
This interpretation can be found in \cite{Burke85}.

So, we take this construction and transpose it to the time-dependent setting 
by using the definitions of the preceding section.

\begin{theorem}
Let $\dotnabla = (\nabla,C,A,B)$ be a time-dependent covariant derivation operator, 
$(t,p_0) \in \R\times M$, 
$v_0,w_0\in \T_{p_0} M$, and 
$\varepsilon$.
Consider:
\begin{itemize}
\item 
$\gamma_1(s)$ the geodesic of~$\dotnabla$ defined by the initial conditions
$\gamma_1(t) = p_0$ and $\gamma_1'(t) = v_0$.
Let $p_1 = \gamma_1(t+\varepsilon)$ and $v_1,w_1\in \T_{p_1}M$ the parallel transports of $v_0$ and $w_0$ along $\gamma_1$ to the point $p_1$.
\item 
$\gamma_2(s)$ the time-dependent geodesic such that 
$\gamma_2(t+\varepsilon) = p_1$ and $\gamma_2'(t+\varepsilon) = w_1$.
Let $p_2 = \gamma_2(t+2\varepsilon)$ and $v_2,w_2\in \T_{p_2}M$ the parallel transports of $v_1$ and $w_1$ to $p_2$.
\item 
$\gamma_3(s)$ the time-dependent geodesic such that 
$\gamma_3(t+2\varepsilon) = p_2$ and $\gamma_3'(t+2\varepsilon) = v_2$.
Let $p_3 = \gamma_3(t+\varepsilon)$ and $w_3\in \T_{p_3}M$ the parallel transport of $w_2$ to $p_3$.
\item 
$\gamma_4(s)$ the time-dependent geodesic such that 
$\gamma_4(t+\varepsilon) = p_3$ and $\gamma_4'(t+\varepsilon) = w_3$.
Let $p_4 = \gamma_4(t)$.
\end{itemize}
Then, defining $c(\varepsilon) = p_4$ for small $\varepsilon$,
$c$ is a smooth path that satisfies:
\begin{itemize}
\item 
$c(0) = p_0$,
\item 
$c'(0) = 0$,
\item 
in coordinates
$\frac 12 c''(0)$ 
has components
$-(\Gamma_{ij}^k - \Gamma_{ji}^k) v^iw^j - (A_i^k-B_i^k)(v^i - w^i)$.
\end{itemize}
\end{theorem}
\begin{proof}
Taylor expansions of the coordinate expressions for the geodesics are the following:
\begin{align*}
\gamma_1^k(s) &= p_0^k + v_0^k \cdot(s - t)
- \frac12 \left(\Gamma_{ij}^k\,v_0^i\,v_0^j + A^k_i\,v_0^i + B^k_i\,v_0^i + C^k\right) (s - t)^2
+ o((s-t)^2)\,,
\\
\gamma_2^k(s) &= p_1^k + w_1^k \cdot(s - t - \varepsilon)
- \frac12 \left(\Gamma_{ij}^k\,w_1^i\,w_1^j + A^k_i\,w_1^i + B^k_i\,w_1^i + C^k\right) (s - t - \varepsilon)^2
+ o((s-t)^2)\,,
\\
\gamma_3^k(s) &= p_2^k + v_2^k \cdot(s - t - 2\varepsilon)
- \frac12 \left(\Gamma_{ij}^k\,v_2^i\,v_2^j + A^k_i\,v_2^i + B^v_i\,v_2^i + C^k\right) (s - t - 2\varepsilon)^2
+ o((s-t)^2)\,,
\\
\gamma_4^k(s) &= p_3^k + w_3^k \cdot(s - t - \varepsilon)
- \frac12 \left(\Gamma_{ij}^k\,w_3^i\,w_3^j + A^k_i\,w_3^i + B^k_i\,w_3^i + C^k\right) (s - t - \varepsilon)^2
+ o((s-t)^2)\,.
\end{align*}
On the other hand, expressions for the parallel transported vectors are:
\begin{align*}
            v_1^k &= v_0^k
                - \left(\Gamma_{ij}^k\,v_0^i\,v_0^j + A^k_i\,v_0^i + B^k_i\,v_0^i + C^k\right) \varepsilon + o(\varepsilon)\,,
            \\
            w_1^k &= w_0^k
                - \left(\Gamma_{ij}^k\,v_0^i\,w_0^j + A^k_i\,v_0^i + B^k_i\,w_0^i + C^k\right) \varepsilon + o(\varepsilon)\,,
            \\
            v_2^k &= v_1^k
                - \left(\Gamma_{ij}^k\,w_1^i\,v_1^j + A^k_i\,w_1^i + B^k_i\,v_1^i + C^k\right) \varepsilon + o(\varepsilon)\,,
            \\
            w_2^k &= w_1^k
                - \left(\Gamma_{ij}^k\,w_1^i\,w_1^j + A^k_i\,w_1^i + B^k_i\,w_1^i + C^k\right) \varepsilon + o(\varepsilon)\,,
            \\
            w_3^k &= w_2^k
                + \left(\Gamma_{ij}^k\,v_2^i\,w_2^j + A^k_i\,v_2^i + B^k_i\,w_2^i + C^k\right) \varepsilon + o(\varepsilon)
            \,.
        \end{align*}
With these equations, it is a straightforward calculation to show that
$$
p_4^k = 
p_0^k - 
\left( 
   (\Gamma_{ij}^k - \Gamma_{ji}^k)\,v^i\,w^j + (A_i^k-B_i^k)(v^i - w^i)
\right)
\varepsilon^2 
+ o(\varepsilon^2)\,,
$$
from which the stated results follow.
\end{proof}

From the coordinate expression we notice that,
in the time-independent case,
we obtain the torsion tensor, with opposite sign.
It is also clear how to write the second derivative in a coordinate-independent 
way.
So we introduce the following definition:
\begin{definition}
The   
\dfn{time-dependent torsion operator}
of a time-dependent covariant derivative operator
$\dotnabla = (\nabla,C,A,B)$
is defined as
\begin{equation}
\label{eq:tdtorsion}
\mathcal{T}^\dotnabla (X, Y) = 
T^\nabla(X,Y) + (A - B)\,(X-Y)
\,,
\end{equation}
where $T^\nabla$ is the torsion tensor of the connection~$\nabla$
and $X,Y$ are arbitrary time-dependent vector fields on~$M$.
\end{definition}

Remarkably, this result can be seen as the vertical part of the torsion 
in the time-extended manifold.

\begin{proposition}
Consider a time-dependent covariant derivation 
$\dotnabla = (\nabla,C,A,B)$
on~$M$,
and $X,Y$ two time-dependent vector fields.
Let $\widehat\nabla$ be any connection on $\R \times M$ that extends $\dotnabla$,
and $T^{\widehat\nabla}$ its torsion tensor.
Let $\widetilde X, \widetilde Y$ be the suspensions of $X,Y$. 
Then
$$
\mathcal{T}^\dotnabla(X,Y) 
= 
\T(\rho_2) \circ T^{\widehat\nabla} (\widetilde X, \widetilde Y)
\,.
$$
\end{proposition}
\begin{proof}
This result follows from the following calculation,
that uses Theorem~\ref{th:nabladotCharacterization}:
\begin{align*}
T^{\widehat\nabla} (\widetilde X, \widetilde Y) 
&=
\widehat\nabla_{\widetilde X} \widetilde Y -
\widehat\nabla_{\widetilde X} \widetilde Y - 
[ \widetilde X, \widetilde Y]
\\
&=
\left( \strut
(\alpha-\beta)(X-Y) + \varepsilon(X,Y)-\varepsilon(Y,X) 
\right) 
\parder{}{t} 
+
T^\nabla(X,Y) + 
(A - B)\,(X-Y)\,.
\end{align*}
\end{proof}

Finally, we will give a third construction of the torsion operator.
\begin{proposition}
Given time-dependent vector fields $X,Y$,
the time-dependent torsion operator on them can be expressed as
\begin{equation}
\mathcal{T}^\dotnabla(X,Y) = 
\dotnabla_XY - \dotnabla_YX - \tdBracket{X, Y}\,
.
\end{equation}
\end{proposition}
\begin{proof}
It is immediate from the expressions of each of its terms:
\eqref{eq:tdtorsion},
\eqref{eq:tdcovariant},
\eqref{eq:tdbracket}.
\end{proof}

In other words, 
the expression for time-dependent torsion operator is the same as for the usual torsion tensor, 
but replacing each term for its time-dependent version.

\medskip

The properties of this operator are somewhat different 
from the usual torsion tensor.
For instance, 
$\mathcal{T}^{\dotnabla}$ is a tensor field
if and only if $A=B$
(moreover, in this case 
$\mathcal{T}^\dotnabla(X, Y)$ matches the value of 
$T^\nabla(X, Y)$).
In addition:
\begin{proposition}
\label{prop:torsionlessEquivalence}
$\mathcal{T}^{\dotnabla} = 0$
if, and only if,
the time-dependent connection $\nabla$ is torsionless and $A=B$.
\end{proposition}
\begin{proof}
We take
$
\mathcal{T}^\dotnabla(X, Y) =
T^{\nabla} (X,Y) + (A{-}B) (X{-}Y)
$.
It is clear that
$T^{\dotnabla} = 0$
if $T^\nabla=0$ and $A=B$.
For the direct implication,
putting $Y=0$ in the right-hand side of the equation yields $A-B=0$.
Then $T^{\nabla} (X,Y) = 0$ 
for every $X,Y$.
\end{proof}

As a final remark,
notice that the time-dependent covariant derivation associated with a 
time-dependent Riemannian metric satisfies $A=B$,
therefore 
$T^\dotnabla$ is tensorial and coincides with $T^\nabla=0$.

\section{Example: double pendulum with variable masses}
\label{sec:example}

Consider a double pendulum in a vertical plane,
with two \emph{time-varying masses} $m_1$ and~$m_2$.
The first mass is connected to a fixed point in the origin of coordinates
through a massless rod with length~$\ell_1$,
and the second mass is connected to the first one through a massless rod with length $\ell_2$.
Let $\phi_1$ and $\phi_2$ be the angle of both rods with respect to an upwards vertical vector. 
Then $(\phi_1,\phi_2)$ are coordinates of the system.

The kinetic energy of the system $T = \dfrac12 g(\dot\phi,\dot\phi)$
is given by the metric with matrix
$$
\mathbf{G} =
\begin{pmatrix}
    \ell_1^2 \,(m_1 {+} m_2) & 
    \ell_1 \ell_2 m_2 \cos(\phi_1{\shortmin}\phi_2) 
    \\
    \ell_1 \ell_2 m_2 \cos(\phi_1{\shortmin}\phi_2) & 
    \ell_2^2 \,m_2 \\
\end{pmatrix}
\,.
$$
From this we find the inverse and the time-derivative of the metric to be
$$
\mathbf{G}^{-1} = 
\frac1{\ell_1^2 \,\ell_2^2 \,m_2 
\left( m_1 + m_2 \sin^2(\phi_1{\shortmin}\phi_2) \right)}
\begin{pmatrix}
    \ell_2^2 \,m_2 & 
    -\ell_1\ell_2 m_2 \cos(\phi_1{\shortmin}\phi_2) 
    \\
    -\ell_1\ell_2 m_2 \cos(\phi_1{\shortmin}\phi_2) & 
    \ell_1^2 \,(m_1{+}m_2) \\
\end{pmatrix}
$$
and
$$
\dot {\mathbf{G}} = 
\frac{\dot m_2}{m_2} \, \mathbf{G}
+ \ell^2_1 \,\frac{\dot m_1 m_2 - m_1\dot m_2}{m_1}
\begin{pmatrix}  1 & 0 \\ 0 & 0\end{pmatrix}
\,.
$$
Their product is
$$
\mathbf{G}^{-1}  \dot {\mathbf{G}} = 
\frac{\dot m_2}{m_2} \, \Id
+ \frac{\dot m_1 m_2 - m_1 \dot m_2}{m_1\left(m_1+ m_2 \sin^2(\phi_1{\shortmin}\phi_2)\right)}
\begin{pmatrix}
    1 & 0 \\
    -\frac{\ell_1}{\ell_2}\cos(\phi_1{\shortmin}\phi_2) & 0
\end{pmatrix}
\,.
$$

With respect to the Christoffel symbols, the non-zero ones are:
\begin{align*}
    \Gamma_{11}^1 &= \frac{m_2 \cos(\phi_1{\shortmin}\phi_2) \sin(\phi_1{\shortmin}\phi_2)}{m_1 + m_2 \sin^2(\phi_1{\shortmin}\phi_2)}\,,
      & 
    \Gamma_{22}^1 &= \frac{\ell_2 m_2 \sin(\phi_1{\shortmin}\phi_2)}{\ell_1 \left(m_1+ m_2 \sin^2(\phi_1{\shortmin}\phi_2)\right)}\,,
    \\
    \Gamma_{11}^2 & = -\frac{\ell_1 \left(m_1 {+}m_2\right) \sin(\phi_1{\shortmin}\phi_2)}{\ell_2 \left(m_1+ m_2 \sin^2(\phi_1{\shortmin}\phi_2)\right)}\,,
      &
    \Gamma_{22}^2 &= -\frac{m_2 \cos(\phi_1{\shortmin}\phi_2) \sin(\phi_1{-}\phi_2)}{m_1 + m_2 \sin^2(\phi_1{\shortmin}\phi_2)}\,.
\end{align*}

In case there exists a conservative force given by a potential function $V$ acting on the system, 
the equations of motion are
$$
\ddot\phi_k + \Gamma_{ij}^k\dot\phi_i\dot\phi_j + g^{kl}\dot g_{li}\dot\phi_i
= -g^{ki}\left[ V\right]_i
\,,
$$
where $\left[V\right]$ is the Euler--Lagrange operator applied to~$V$.
When $V$ is zero we obtain the geodesics as defined by
equation~\eqref{eq:geodesic-coordinates},
\begin{align*}
\ddot\phi_1 &=
- \frac{m_2 \sin(\phi_1{\shortmin}\phi_2)}{\ell_1\left(m_1 + m_2\sin^2(\phi_1{\shortmin}\phi_2)\right)}
        \left(
            \ell_1\cos(\phi_1{\shortmin}\phi_2)\dot\phi_1^2 + \ell_2\dot\phi_2^2
        \right) \\
&\quad - \left(
        \frac{\dot m_2}{m_2}
        + \frac{\dot m_1 m_2 - m_1 \dot m_2}{m_1\left(m_1+ m_2 \sin^2(\phi_1{\shortmin}\phi_2)\right)}
    \right) \dot\phi_1\,,
\\
\ddot\phi_2 &=
\frac{\sin(\phi_1{\shortmin}\phi_2)}{\ell_2\left(m_1 + m_2\sin^2(\phi_1{\shortmin}\phi_2)\right)}
        \left(
            \ell_1(m_1 + m_2)\dot\phi_1^2 + \ell_2 m_2 \cos(\phi_1{\shortmin}\phi_2)\dot\phi_2^2
        \right)
- \frac{\dot m_2}{m_2}\dot\phi_2
\end{align*}

Lastly we compute the time-derivative of the 
Levi-Civita connection,
which turns out to be non-zero.
Denoting by $W = m_1 \dot m_2 - m_1 \dot m_2$,
we have:
\begin{multline*}
\dot\Gamma 
=
\frac
{W \sin(\phi_1 {\shortmin} \phi_2)}
{\,\ell_1 \ell_2  
\left( m_1 + m_2 \sin^2(\phi_1{\shortmin}\phi_2) \right)^2\,}
\Bigg( 
\ell_1 \ell_2 \cos(\phi_1 {\shortmin} \phi_2) 
\left(
\parder{}{\phi_1} \otimes \d \phi_1 \otimes \d \phi_1 -
\parder{}{\phi_2} \otimes \d \phi_2 \otimes \d \phi_2
\right)
+
\\
-
\ell_1^2 \cos^2(\phi_1 {\shortmin} \phi_2) \,
\parder{}{\phi_2} \otimes \d \phi_1 \otimes \d \phi_1 
+
\ell_2^2 \,
\parder{}{\phi_1} \otimes \d \phi_2 \otimes \d \phi_2
\Bigg)\,.
\end{multline*}

\section{Conclusions and further research}

The main goal of this paper was to develop a suitable framework for the geometric description of geodesics associated with a time-dependent Riemannian metric. 
To this end, we introduced the notions of time-dependent Riemannian metrics and time-dependent connections. 
We derived the equations for geodesics through two different approaches: 
by extremizing the energy functional associated with a time-dependent Riemannian metric, 
and by projecting the geodesic equations of a suitable connection defined on the product manifold $\mathbb{R} \times M$. 
Along the way, we introduced the concept of time-dependent covariant derivation operators, 
their action on vector fields, the notion of parallel transport, and the torsion operator associated with such connections.

There are several directions for further research. A natural next step is to develop the notions of curvature for the time-dependent covariant derivation operator. 
Another interesting subject is the study of the second variation of the energy functional, which could yield information about the stability of geodesics in this time-dependent setting. 
Finally, it would be valuable to analyze more examples of physical interest, such as time-dependent metrics arising in general relativity or analytical mechanics.

\appendix 
\section{Appendix: products and pullbacks}\label{app:products}

In this appendix we gather some elementary results as a reminder for the reader and also to fix notation.

Let $\pi \colon E \to M$ be a vector bundle,
and $f \colon N \to M$ a map.
The \emph{pullback} of~$\pi$ by~$f$ is the vector bundle 
$f^*(E) \to N$
whose total space is the fibre product 
$
N \times_f E = 
\{(y,v) \in N \times E \mid f(x) = \pi(v) \}
$
and whose fibres are 
$f^*(E)_y = E_{f(y)}$.
A section of $f^*(E)$ is identified with a map
$\sigma \colon N \to E$ such that 
$\pi \circ \sigma = f$,
that is,
a section of~$E$ along~$f$.

This is of particular interest when the vector bundle is the tangent bundle
$\tau_M \colon \T M \to M$ of~$M$,
whose (smooth) sections are the 
$\Cinfty(M)$-module of the 
tangent vector fields on~$M$,
$\vf(M) = \Sec(\T M)$.
So, a \emph{vector field along} 
$f \colon N \to M$
is just a map 
$V \colon N \to \T M$ 
such that, for every $y \in N$, $V_y \in \T_{f(y)}M$.
The set of such maps will be denoted as 
$\vf(f)$, and is a $\Cinfty(N)$-module.
Prominent examples of vector fields along maps are 
the velocity $\gamma'$ of a path $\gamma \colon I \to M$,
the Frenet frame in Riemannian geometry, 
or variation fields in variational calculus.

A connection on a manifold is a linear connection on its tangent bundle.
We assume that these concepts are well-known to the reader.
Nevertheless, let us remark that
if one has a map $f \colon N \to M$
and there is a linear connection on a vector bundle $E \to M$,
this induces a linear connection on the pullback bundle 
$f^*(E) \to N$.
The most important case is that of 
a connection on a manifold~$M$ and a map (path) 
$\gamma \colon \R \to M$;
then the covariant derivative of 
$w \in \vf(\gamma) \cong \Sec(\gamma^*(\T M))$
with respect to the unit vector field of~$\R$,
$\nabla_{\!\totder{}{t}} w$,
will be simply denoted as
$\nabla_{\!t} w$;
this is the well-known notion of covariant derivation along a path.

Now we consider a product manifold and its projections,
$\rho_1 \colon M_1 \times M_2 \to M_1$ 
and analogously $\rho_2$.
Given any $x = (x_1,x_2) \in M_1 \times M_2$
we have an isomorphism
$$
(\T_x \rho_1,\T_x \rho_2) \colon
\T_{x} (M_1 \times M_2) \xlongrightarrow{\cong}
\T_{x_1}M_1 \times \T_{x_2}M_2 \cong 
\T_{x_1}M_1 \oplus \T_{x_2}M_2 
\,;
$$
we can write the elements of this tangent space 
as ordered pairs or as sums of vectors of each subspace.
When we carry this to sections,
a vector field $X$ on $M_1 \times M_2$ 
(for instance)
can be uniquely written as 
$X = X_1 + X_2$,
where $X_i \in \vf(\rho_i)$
---indeed $X_i = \T \rho_i \circ X$.
Therefore
$$
\vf(M_1 \times M_2) = \vf(\rho_1) \oplus \vf(\rho_2) \,.
$$
In particular, 
any vector field on $M_i$ yields in a trivial way a vector field on the product.

We are especially interested in the case $\R \times M$.
The unit vector field $\totder{}{t}$ of $\R$
yields a vector field in the product that we will denote simply as 
$\parder{}{t}$.
A \emph{time-dependent vector field} on~$M$ is a vector field $X$ along 
$\rho_2 \colon \R \times M \to M$,
and therefore it defines a vector field on the product.
Then we can construct the \emph{suspension} of~$X$ as the vector field
$$
\widetilde X = \parder{}{t} + X\,.
$$

Along the paper we make use of a number of sections of vector bundles $E \to M$ depending on a parameter $t \in \R$,
$\sigma_t$
that is,
sections $\sigma$ along $\rho_2 \colon \R \times M \to M$,
with $\sigma_t = \sigma(t,\cdot)$.
So each vector space $E_x$ we have a path 
$\sigma(\cdot,x) \colon 
t \mapsto E_x$.
This path has a derivative
(in the sense of basic calculus)
that we will denote by $\dot\sigma(t,x)$;
thus 
$\dot\sigma$ is also a section along~$\rho_2$.
It is clear that, in vector bundle coordinates,
the components of $\dot\sigma$
are computed simply as the partial derivatives with respect to~$t$.
The simplest example is the partial time derivative 
$\dot f = \parder{f}{t}$ of a function.
Then,
if $X$ is a time-dependent vector field on~$M$,
$$
\Lie_{\widetilde X} f = \dot f + \Lie_X f\,.
$$
As another example, one can easily check the Lie bracket
$\displaystyle
\left[ \parder{}{t} , X \right] = \dot X
$.
The right-hand side of this equality belongs to 
$\vf(\rho_2)$,
while the left-hand side is computed in
$\vf(\R \times M)$.

\section{Appendix: flow of a time-dependent vector field}\label{app:time-dep-flow}

In this section we recall some properties of time-dependent flows.
We omit most details
---see for instance \cite{Lee_2013,MR_07}.

If $X$ is a time-dependent vector field on~$M$,
we denote by 
$
\Flow_X(t,t_\circ,p) 
$
its \emph{time-dependent flow},
that is, 
the integral curve of $X$ with initial condition $(t_\circ,p)$.
This time-dependent flow is closely related to 
the ordinary flow $\widetilde \Flow_{\widetilde X}$ of the suspension $\widetilde X$ of~$X$. 
Indeed,
$\widetilde \Flow_{\widetilde X}(t-t_\circ;t_\circ,p) = (t,\Flow_X(t,t_\circ,p))$.

Its main properties are:
\begin{equation}
\partial_1 \Flow_X(t,t_\circ,p) = X(t,\Flow_X(t,t_\circ,p))
\,,
\end{equation}
where 
$
\partial_1 \Flow_X(t,t_\circ,p)
$
denotes the tangent vector with respect to the first variable,
and the group law 
\begin{equation}
\Flow_X^{t,t} = \Id_M
\,, \qquad 
\Flow_X^{t'',t'} \circ \Flow_X^{t',t} = \Flow_X^{t'',t}
\,,
\end{equation}
where $\Flow^{t,s} = \Flow(t,s,\cdot)$.

Similar to $\partial_1 \Flow_X$ we are interested in the tangent vector
$\partial_2 \Flow_X$ with respect to the second scalar variable.
It can be expressed as
$$
\partial_2 \Flow_X(t,t_\circ,p) =
- \T \Flow_X^{t,t_\circ}(p) \cdot X(t_\circ, p)
\,.
$$

Consider a time-dependent function
$f \colon \R \times M \to \R$.
Analogously to the Lie derivative,
we can evaluate its rate of change along integral curves of~$X$;
the result is
$$
\restr{\frac{\d}{\d t}}{t=t_1} f(t,\Flow_X^{t, t_\circ}(p)) =
(\Flow_X^{t_1, t_\circ})^*
\left(
\restr{\parder{f}{t}}{t=t_1} \!\!\!\! + \Lie_{X} f
\right) (t_1,p)
\,,
$$
where 
$\Lie_{X} f$
is the ordinary Lie derivative but with frozen time.
A similar expression can be given for the derivative of a 
time-dependent tensor field.

\section*{Acknowledgements}

We acknowledge the financial support of the 
Ministerio de Ciencia, Innovaci\'on y Universidades (Spain), 
projects PID2021-125515NB-C21, and RED2022-134301-T of AEI,
and the
Ministry of Research and Universities of the Catalan Government, 
project 2021 SGR 00603, {Geometry of Manifolds and Applications} (GEOMVAP).

We dedicate this article to the memory of our beloved colleague 
Prof.\ Miguel C. Mu\~noz-Lecanda, 
who undoubtedly would have illuminated this work with his comments and wisdom.


\bibliographystyle{abbrv}
{
\small
\bibliography{references.bib}
}

\end{document}